\documentclass[a4paper,10pt,twoside]{article}
\pdfoutput=1

\usepackage[activate={true,nocompatibility},final,tracking=true,kerning=true,spacing=true,factor=1100,stretch=0,shrink=0]{microtype}
 \SetTracking{encoding={*}, shape=sc}{30}
 \microtypecontext{spacing=nonfrench}

\usepackage[T1]{fontenc} 
\usepackage[utf8]{inputenc}
\usepackage[british]{babel}
\usepackage{mathrsfs}
\usepackage{amsfonts}
\usepackage{amsmath} 
\usepackage{amsthm}  
\usepackage{amssymb}
\usepackage{mathtools}
\usepackage{graphicx}
\usepackage{xcolor}
\usepackage{url}
\usepackage{cite}
\usepackage{xifthen}
\usepackage{stmaryrd}
\usepackage{titling}
\usepackage{centernot}
\usepackage{indentfirst}
\usepackage[affil-it]{authblk}
\usepackage{hyperref}
\hypersetup{colorlinks=false, pdfborder={0 0 0}}
\usepackage{tikz}
\usetikzlibrary{matrix,arrows,decorations.pathmorphing, shapes, decorations.fractals, patterns, positioning}
\tikzset{snake it/.style={decorate, decoration={snake, segment length=5.3pt, amplitude=0.25mm}}}
\usepackage[inline]{enumitem}
\setdescription{font=\normalfont\itshape}     
\usepackage{fancyhdr}
\newcommand{\into}{\hookrightarrow}
\newcommand{\implica}{\rightarrow}
\newcommand{\coimplica}{\leftrightarrow}
\renewcommand{\phi}{\varphi}
\renewcommand{\models}{\vDash}
\newcommand{\proves}{\vdash}
\newcommand{\pfin}{{\mathscr P}_{\mathrm{fin}}}
\newcommand{\restr}{\upharpoonright}
\newcommand*\bigcircled[1]{\tikz[baseline=(char.base)]{
    \node[shape=circle,draw,inner sep=1pt] (char) {#1};}} 
\newcommand{\monster}{\mathfrak U}
\newcommand{\smallsubset}{\mathrel{\subset^+}}
\newcommand{\smallprec}{\mathrel{\prec^+}}
\newcommand{\satext}{\mathrel{^+\!\!\succ}}
\newcommand{\invtypes}{S^{\textnormal{inv}}}
\newcommand{\invext}{\mid}
\newcommand{\bla}[4]{{#1}_{#2}#3\ldots#3{#1}_{#4}}
\newcommand{\wort}{\mathrel{\perp^{\!\!\mathrm{w}}}}
\newcommand{\nwort}{\mathrel{{\centernot\perp}^{\!\!\mathrm{w}}}}
\newcommand{\invtilde}{\operatorname{\widetilde{Inv}}(\monster)}
\newcommand{\invtildeof}[1]{\operatorname{\widetilde{Inv}}({#1})}
\newcommand{\equidom}{\mathrel{\equiv_\mathrm{D}}}
\newcommand{\nequidom}{\mathrel{{\centernot\equiv}_{\mathrm{D}}}}
\newcommand{\domeq}{\mathrel{\sim_\mathrm{D}}}
\newcommand{\doms}{\mathrel{\ge_\mathrm{D}}}
\newcommand{\ndoms}{\mathrel{{\centernot\ge}_{\mathrm{D}}}}
\newcommand{\pow}[2]{#1^{(#2)}}
\newcommand{\meet}{\sqcap}
\newcommand{\lmt}{L_\mathrm{mt}}
\newcommand{\tdmt}{\mathsf{DMT}}
\DeclareMathOperator{\Th}{Th}
\DeclareMathOperator{\tp}{tp}
\DeclareMathOperator{\dcl}{dcl}
\DeclareMathOperator{\aut}{Aut}
\DeclarePairedDelimiter{\set}{\{}{\}}
\DeclarePairedDelimiter{\abs}{\lvert}{\rvert}
\DeclarePairedDelimiter{\class}{\llbracket}{\rrbracket}

\theoremstyle{definition}
\newtheorem{defin}{Definition}[section]
\newtheorem{thm}[defin]{Theorem}
\newtheorem{pr}[defin]{Proposition}
\newtheorem{co}[defin]{Corollary}
\newtheorem{lemma}[defin]{Lemma}
\newtheorem{notation}[defin]{Notation}
\newtheorem{eg}[defin]{Example}
\newtheorem{rem}[defin]{Remark}
\newtheorem{fact}[defin]{Fact}
\newtheorem{ass}[defin]{Assumption}
\newtheorem*{claim}{Claim}
\newtheorem{alphthm}{Theorem}

\let\oldqed\qedsymbol
\newcommand{\qedclaim}{\mbox{$\underset{\textsc{claim}}{\oldqed}$}}
\makeatletter
\newenvironment{claimproof}[1][\it Proof of Claim]{
\let\qedsymbol\qedclaim
  \par
  \pushQED{\qed}%
  \normalfont \topsep6\p@\@plus6\p@\relax
  \trivlist
\item[\hskip\labelsep
  \upshape
  #1\@addpunct{.}]\ignorespaces
}{%
  \popQED\endtrivlist\@endpefalse
}
\makeatother
\let\qedsymbol\oldqed

\makeatletter
\newcommand{\subjclass}[2][2020]{%
  \let\@oldtitle\@title%
  \gdef\@title{\@oldtitle\footnotetext{#1 \emph{Mathematics subject classification.} #2}}%
}
\newcommand{\keywords}[1]{%
  \let\@@oldtitle\@title%
  \gdef\@title{\@@oldtitle\footnotetext{\emph{Key words and phrases.} #1.}}%
}
\makeatother

\author{Rosario Mennuni%
  \thanks{email: \url{R.Mennuni@tutanota.com} \textsc{orcid}: \url{https://orcid.org/0000-0003-2282-680X}}}
\affil{University of Leeds}
\title{Weakly binary expansions of dense meet-trees}
\keywords{domination, domination monoid, invariant types, meet-tree, semilinear order, weak binarity}
\subjclass{Primary: 03C45. Secondary: 06A12.}

\begin{document}
\maketitle
\begin{abstract}
  We compute the domination monoid in the theory $\tdmt$ of dense meet-trees. In order to show that this monoid is well-defined, we  prove \emph{weak binarity} of $\tdmt$ and,  more generally, of certain expansions of it by binary relations on sets of open cones, a special case being the theory $\mathsf{DTR}$ from~\cite{estevankaplan}. We then describe the domination monoids of such expansions in terms of those of the expanding relations.
\end{abstract}
If asked what a \emph{tree} is, a mathematician has a number of options to choose from. The graph theorist's answer will probably contain the words ``acyclic'' and ``connected'', while the set theorist may have in mind certain sets of sequences of natural numbers. In this paper we are instead concerned with   \emph{lower semilinear orders}: posets where the set of predecessors of each element is linearly ordered.

More specifically, a \emph{meet-tree} is a lower semilinear order $<$ in which each pair of elements $a,b$ has a greatest common lower bound, their \emph{meet} $a\meet b$. When viewed as $\set{<, \meet}$-structures, finite meet-trees form an amalgamation class, hence have a Fra\"iss\'e limit, the universal homogeneous countable meet-tree. Its complete first-order theory $\tdmt$ is that of \emph{dense meet-trees}: dense lower semilinear orders with everywhere infinite ramification.

Such structures have received a certain amount of model-theoretic attention in the recent (and not so recent) past. They appear in the classification of countable $2$-homogeneous trees from~\cite{droste}, and have since been  important in the theory of permutation groups, see for instance~\cite{treelike,dhmdmt,relrelbet}. 
  More recently, they were shown to be dp-minimal in~\cite{simondmt}, and the automorphism group of the unique countable one was studied in~\cite{autdmt}, while the interest in similar structures goes back at the very least to~\cite{peretyiatkin, woodrowtrees}, where they were used as a base to produce examples in the context of Ehrenfeucht theories.  
Here we study $\tdmt$, and some of the expansions defined in~\cite{estevankaplan}, from the viewpoint of domination, in the sense of~\cite{invbartheory}.

One motivation for such a study comes from valuation theory. The nonzero points of a valued field $K$ can notoriously be identified with the \emph{branches} of a meet-tree, that is, its maximal linearly ordered subsets. This identification is used, for instance, to endow $K$ with a C-relation; see~\cite{hollyvftrees, ominvar}. Viewing the residue field $k$ of $K$  as a set of open valuation balls yields a correspondence between $k$  and, for an arbitrary but fixed point $g$ of the underlying tree, the set of \emph{open cones} above $g$: the equivalence classes of the relation $E(x,y)\coloneqq x\meet y>g$ defined on the set of points above $g$. If $K$ has pseudofinite residue field, then $k$ interprets a structure  elementarily equivalent to the Random Graph (see~\cite{duret,rhgpsf}); it is therefore interesting to study the theory of a dense meet-tree  with a Random Graph structure on each set of open cones above a point. This theory was used  in~\cite{estevankaplan}, where it was dubbed $\mathsf{DTR}$,   to show that restrictions  to nonforking bases need not preserve  $\mathsf{NIP}$.

Another motivation is rooted in the study of \emph{invariant types}: types over a saturated model $\monster$ of a first order theory which are fixed, under the natural action of $\aut(\monster)$ on the space $S(\monster)$ of types, by the stabiliser of some small set. The space of invariant types is a semigroup when equipped with the tensor product, and can be endowed with the preorder of \emph{domination}, where a type $p(x)$ dominates a type $q(y)$ iff $q(y)$ is implied by the union of $p(x)$ with a small type $r(x,y)$ consistent with $p(x)\cup q(y)$. The induced equivalence relation, called  \emph{domination-equivalence}, may or may not be a congruence with respect to the tensor product, and  some conditions ensuring this to be the case were isolated in~\cite{invbartheory}.  One of the main results in the present work is a proof that one of them, \emph{weak binarity} (Definition~\ref{defin:wb}), is satisfied by $\tdmt$, and by certain expansions of the latter by binary structures on sets of open cones, a special case of which is $\mathsf{DTR}$. This guarantees the semigroup operation to descend to the quotient, so we may proceed to calculate the \emph{domination monoid} $\invtilde$.

The  paper is structured as follows. After briefly reviewing standard definitions and facts about dense meet-trees and invariant types in Section~\ref{sec:prelim}, we recall in  Section~\ref{sec:wbin} the definition of \emph{weak binarity} and prove that, despite not being binary, $\tdmt$ and all of its \emph{binary cone-expansions}  
(Definition~\ref{defin:binconexp}) are weakly binary. This is in particular the case for  
$\mathsf{DTR}$.
\begin{alphthm}[Theorem~\ref{thm:expdmt}]
  The theory of dense meet-trees is weakly binary, and so is each of its  binary cone-expansions.
\end{alphthm}
\noindent Hence the monoid $\invtilde$ is well-defined in such theories, and we  proceed to compute it. The case of pure dense meet-trees is handled in Section~\ref{sec:invtildedmt}. 
\begin{alphthm}[Theorem~\ref{thm:dmt}]
  If $\monster$ is a monster model of the theory of dense meet-trees, then there is a set $X$ such that  \[\invtilde\cong\pfin(X)\oplus\bigoplus_{g\in \monster} \mathbb N\] where $\pfin(X)$ is the  upper semilattice of finite subsets of $X$. 
\end{alphthm}
\noindent In the same section, we take the opportunity to record an instance of a theory where domination differs from $\mathrm{F}^\mathrm{s}_{\kappa}$-isolation in the sense of Shelah, Example~\ref{eg:semineeded}.
Theorem~\ref{thm:dmt} is generalised in Section~\ref{sec:invtildeexp} to \emph{purely binary cone-expansions} (Definition~\ref{defin:binconexp}), such as $\mathsf{DTR}$. If $\monster$ is a monster model of such an expansion, $X$ is given by applying Theorem~\ref{thm:dmt} to the underlying tree and, for  $g\in \monster$, we denote by $O_g$ the structure on the set of open cones above $g$, we obtain the following.
\begin{alphthm}[Theorem~\ref{thm:invtildexp}]
  If $T$ is a purely binary cone-expansion of $\tdmt$,  \[\invtilde \cong \pfin(X)\oplus\bigoplus_{g\in \monster} \invtildeof{O_g}\]
\end{alphthm}

\paragraph{Acknowledgements}
A large part of this paper first appeared in the author's PhD thesis~\cite{mythesis},  supervised by Dugald Macpherson and Vincenzo Mantova. The material and presentation has been greatly improved by their comments and suggestions, on~\cite{mythesis} and on drafts of this manuscript; in particular, Definition~\ref{defin:tamefla} was suggested by Mantova. I would like to thank Itay Kaplan for bringing to my attention cone-expansions of dense meet-trees, and Predrag Tanovi\'c for the fruitful discussions on the global Rudin--Keisler preorder, which inspired Example~\ref{eg:semineeded}. Part of this research has been funded by a Leeds Anniversary Research Scholarship.

 \section{Preliminaries}\label{sec:prelim}
 In what follows, lowercase Latin letters may denote \emph{tuples} of variables or elements of a model. The length of a tuple is denoted by $\abs \cdot$, and its coordinates  will be denoted by subscripts, starting with $0$; we may write, for example,  $a=(\bla a0,{\abs a-1})\in M^{\abs a}$ or, with abuse of notation, simply $a\in M$. Concatenation is denoted by juxtaposition, and elements of a sequence of tuples by superscripts.   For instance, if we write $a=a^0a^1$ then  $a_{\abs{a^0}}$ equals $a^1_0$, the first element of $a^1$. Tuples may be treated as sets, in which case juxtaposition  denotes union, as in $Ab=A\cup \set{b_i\mid i<\abs b}$. \emph{Type} means ``complete type in finitely many variables''.

   Proofs regarding trees have a tendency to split in cases and subcases. As they become much easier to follow if the objects in them are drawn as soon as they appear in the proof, the reader is encouraged to reach for writing devices, preferably capable of producing different colours.
 \subsection{Invariant types}
Fix a complete first-order theory $T$ with infinite models,  a  sufficiently large cardinal $\kappa$, and  a $\kappa$-saturated and $\kappa$-strongly homogeneous $\monster\models T$. \emph{Small} means ``of cardinality strictly less than $\kappa$''; if $A$ is a small subset of $\monster$, we denote this by $A\smallsubset \monster$, or $A\smallprec \monster$ if additionally $A\prec \monster$.   \emph{Global type} means  ``type over $\monster$''.

 \begin{defin}\label{defin:invtp}\*
  \begin{enumerate}
  \item Let $A\subseteq B$.  A  type $p(x)\in S(B)$ is \emph{$A$-invariant}  iff for all $\phi(x;y)\in L$ and $a\equiv_A b$ in $B^{\abs y}$ we have $p(x)\proves \phi(x;a)\coimplica \phi(x;b)$. A global type $p(x)\in S(\monster)$ is \emph{invariant} iff it is $A$-invariant for some $A\smallsubset \monster$. Such an $A$ is called a \emph{base} for $p$.
  \item  If $p(x)\in S(\monster)$ is $A$-invariant and $\phi(x;y)\in L(A)$, write 
\[
(d_p\phi(x;y))(y)\coloneqq\set{\tp_y(b/A)\mid \phi(x;b)\in p, b\in \monster}
\]
The map $\phi\mapsto d_p\phi$ is called the \emph{defining scheme} of $p$ over $A$.
\item We denote by $\invtypes_x(\monster, A)$ the space of global $A$-invariant types in variables $x$, with $A$ small, and by $\invtypes_x(\monster)$ the union of all $\invtypes_x(\monster, A)$ as $A$ ranges among small subsets of $\monster$.  Denote by  $S(B)$ the union of all spaces of types over $B$ in a finite tuple of variables; similarly for, say,  $\invtypes(\monster)$.
  \end{enumerate}
\end{defin}

If we  say that a type $p$ is invariant, and its domain is not specified and not clear from context, it is usually a safe bet to assume that $p\in S(\monster)$. Similarly if we say that a tuple has invariant type without specifying over which set.

 \subsection{Dense meet-trees}
A  poset $(M, <)$ is a \emph{lower semilinear order} iff every pair of elements from each set of the form $\set{x\in M\mid x< a}$ is comparable. Let $\lmt=\set{<, \meet}$, where $<$ is a binary relation symbol and $\meet$ is a binary function symbol. A \emph{meet-tree} is an $\lmt$-structure $M$ such that  $(M,<)$ is a lower semilinear order where every pair of elements $a,b$ has a greatest common lower bound,  their \emph{meet}  $a\meet b$. If $M$ is a meet-tree and $g\in M$,  classes of the equivalence relation defined on $\set{x\in M\mid x> g}$ by $E(x,y)\coloneqq x\meet y>g$ are called \emph{open cones above $g$}.

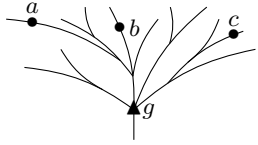
\begin{figure}[h]
  \centering
  \begin{tikzpicture}[scale=0.8, font=\small]
    \node (g) at (0,0) {$\phantom{g}\blacktriangle g$};
    \draw (0,-0.5)--(0,0);
    \draw (0,0) edge[bend left=15] coordinate[pos=0.3] (a) (2,1);
    \draw (a) edge[bend left=15] coordinate  [pos=0.3](b) node [pos=0.9]{$\bullet$} node [pos=0.9,above] {$c$}(1.8,1.3);
    \draw (b) edge[bend right=15]  (1.4,1.5);
    \draw (0,0) edge[bend left=15] coordinate[pos=0.6] (d) (1.2,1.7);
    \draw (d) edge[bend right=15]  (0.6,1.7);
    \draw (0,0) edge[bend right=15] coordinate[pos=0.3] (e) node [pos=0.8]{$\bullet$} node [pos=0.8,right] {$b$}(-0.4,1.7);
    \draw (e) edge [bend right=15] coordinate [pos=0.2] (f) coordinate [pos=0.6] (h) (-1.2, 1.5);
    \draw (e) edge [bend left=15] (0.4, 1.5);
    \draw (f) edge [bend right=15] node [pos=0.8]{$\bullet$} node [pos=0.8,above] {$a$} (-2.1, 1.5);
    \draw (h) edge [bend left=15] (-0.8, 1.7);
    \draw (0,0) edge[bend right=15] coordinate[pos=0.3] (i) (-1.8,0.7);
    \draw (i) edge [bend left=15] (-1.2, 1);
  \end{tikzpicture}
  \caption{the point $a$ is in the same open cone above $g$ as the point $b$, while $c$ is in a different open cone above $g$.}
  \label{figure:cones}
\end{figure}
Finite meet-trees are well-known to form a Fra\"iss\'e class, hence have a Fra\"iss\'e limit, whose theory is complete and eliminates quantifiers.\footnote{For basic Fra\"iss\'e theory, see for instance~\cite[Chapter~7]{hodges}.}
  A \emph{dense meet-tree} is a model of the theory $\tdmt$ of the Fra\"iss\'e limit of finite meet-trees. It is also well-known that this theory can be axiomatised as follows.
\begin{fact}
The theory $\tdmt$ of dense meet-trees is axiomatised by saying that
\begin{enumerate}
\item $(M, <, \meet)$ is a meet-tree;
\item for every $a\in M$, the structure $(\set{x\in M\mid x<a},<)$ is a dense linear order with no endpoints; and
\item for every $a\in M$, there are infinitely many open cones above $a$.
\end{enumerate}
\end{fact}

The following remark will be used throughout, sometimes tacitly.

\begin{rem}\label{rem:meetfacts}
The operation $\meet$ is associative, idempotent, and commutative.  Using this and  quantifier elimination, and observing for example that for every $a, b$ the set defined by $x\meet a=b$ is either empty or  infinite, it is easy to see that the definable closure $\dcl(A)$ of a set $A$ coincides with its closure under meets.  In particular, if $A$ is finite, then so is $\dcl(A)$: by the properties of $\meet$ we just pointed out, its size cannot exceed that of the powerset of $A$.\footnote{While sufficient for our purposes, this upper bound is very far from optimal: one can show that $\abs{\dcl(A)}\le 2\abs A$. See~\cite[Remark~4.6]{estevankaplan}.}
\end{rem}
When working in expansions of $\tdmt$, we will denote the closure of a set $A$ under meets by $\dcl^{\lmt}(A)$. This is justified by the previous remark.
 \begin{defin}
Define the \emph{cut}  $C_p$ of a type $p(x)\in S_1(M)$ to be $\set{c\in M\mid p\proves x\ge c}$ and the cut in $M$ of an element  $b\in N \succ M$ in some elementary extension of $M$ to be $C_b^M\coloneqq C_{\tp(b/M)}$. We say that $C_p$  is \emph{bounded} iff it is bounded from above in $M$.
\end{defin}
 This usage of the word ``cut'' is a bit more general than the one traditionally used for linear orders: our cuts have no upper part, only a lower one.

 It can be shown by standard techniques that $\tdmt$ is $\mathsf{NIP}$,  in fact dp-minimal. This makes it amenable to an analysis of invariant types using indiscernible sequences, and  it turns out that invariant $1$-types are necessarily of one of the six kinds below, as shown by using \emph{eventual types}  (see~\cite[Subsection~2.2.3]{simon}). We refer the reader to~\cite{simondmt} and~\cite[Subsection~2.3.1]{simon}. Alternatively, it is possible to prove this directly via quantifier elimination by considering, for a fixed $p(x)\in \invtypes_1(\monster)$,  what are the possible  values of each $d_p\phi$, as $\phi(x;y)$ ranges among $L$-formulas.

\begin{defin}
Let $\monster\models \tdmt$ and $p(x)\in S_1(\monster)$. We say that $p$ is of kind
  \begin{itemize}
  \item [(0)] iff $p$ is realised in $\monster$;
  \item [(Ia)] iff there is a small linearly ordered set $A\smallsubset \monster$ such that $p(x)\proves\set{x<a\mid a\in A}\cup\set{x> b\mid b\in \monster, b< A}$;
  \item [(Ib)] iff there is a small linearly ordered set $A\smallsubset \monster$ with no maximum such that $p(x)\proves\set{x>a\mid a\in A}\cup\set{x< b\mid b\in \monster, b> A}$, or  there are $a$ and $c$ in $\monster$ such that $p(x)\proves \set{a<x<c}\cup \set{x<b\mid b\in \monster, a<b<c}$;
  \item [(II)] iff there is $g$ such that $p(x)\proves \set{x>g}\cup\set{x\meet b=g\mid b\in \monster, b>g}$;
  \item [(IIIa)] iff $p(x)\proves\set{x\centernot \le b\mid b\in \monster}$ and there is $c\in \monster$ such that $\tp(x\meet c/\monster)$ is of kind (Ia);
  \item [(IIIb)] iff $p(x)\proves\set{x\centernot \le b\mid b\in \monster}$ and there is $c\in \monster$ such that $\tp(x\meet c/\monster)$  is of kind (Ib).
  \end{itemize}
\end{defin}
So types of kind (0), (Ia), or (Ib) correspond to cuts in a linearly ordered  subset of the tree, where in kind (Ib), if the cut of $p$ has a maximum $a$, we are  specifying an existing open cone above $a$.   Kinds (II), (IIIa), and (IIIb) are the corresponding ``branching'' versions. Kind (II) is the type of an element in a new open cone above an existing point. See Figure~\ref{figure:dmtkindsoftypes}.

\begin{figure}[b]
  \centering
  \begin{tikzpicture}[scale=0.8, font=\small]
    \node (g) at (0,0) {$\blacktriangle$};
    \draw[dotted] (0,-1.5) edge node [pos=0.6](x) {$\bullet$} node [pos=0.6, right]{$x$ (Ib)} (0,-0.5);    
    \draw (0,-0.5)--(0,0);
    \draw [dotted] (0,-1.2) edge [bend right=15] node [pos=0.7] {$\bullet$} node [pos=0.7, above]{\phantom{ (IIIb)}$z$ (IIIb)} (-2.7, -0.1);
    \draw [dotted] (0,0) edge[bend left=15] node [pos=0.7] {$\bullet$} node [pos=0.7, above]{\phantom{ (II)}$y$ (II)} (3,0.4);
    \draw (0,-0.3) edge[bend left=15]  (2,0.25);
    \draw (0,0) edge[bend left=15] coordinate[pos=0.3] (a) (2,1);
    \draw (a) edge[bend left=15] coordinate  [pos=0.3](b) (1.7,1.2);
    \draw (b) edge[bend right=15]  (1.4,1.5);
    \draw (0,0) edge[bend left=15] coordinate[pos=0.6] (d) (1.2,1.7);
    \draw (d) edge[bend right=15]  (0.6,1.7);
    \draw (0,0) edge[bend right=15] coordinate[pos=0.3] (e) (-0.4,1.7);
    \draw (e) edge [bend right=15] coordinate [pos=0.2] (f) coordinate [pos=0.6] (h) (-1.2, 1.5);
    \draw (e) edge [bend left=15] (0.4, 1.5);
    \draw (f) edge [bend right=15] (-2.1, 1.5);
    \draw (h) edge [bend left=15] (-0.8, 1.7);
    \draw (0,0) edge[bend right=15] coordinate[pos=0.3] (i) (-1.8,0.7);
    \draw (i) edge [bend left=15] (-1.2, 1);
    \draw (0, -3) -- (0,-1.5);
    \draw (0,-2.7) edge [bend left=15] coordinate [midway] (j) (1.5,-1.7);
    \draw (j) edge [bend right=15] (1.2, -1.2);
    \draw (0, -2.5) edge [bend right=15] (-1.2, -1);
    \foreach \Point in {0,...,11}{\node at (0,-1.5-(\Point*\Point*0.011) {$\blacktriangle$};}
  \end{tikzpicture}
  \caption{some nonrealised $B$-invariant types, where points of $B$ are denoted by triangles. In this picture, the set of triangles below $x$ has no maximum, solid lines lie in $\monster$, and dotted lines lie in a bigger $\monster_1\satext \monster$.  The type of $x$ is of kind (Ib), that of $y$  of kind (II), and that of $z$ of kind (IIIb).}
  \label{figure:dmtkindsoftypes}
\end{figure}
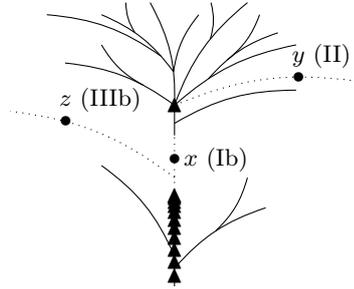

We conclude this section by recording some easy observations for later use.
\begin{lemma}\label{lemma:cutremarks3}\*
  \begin{enumerate}
  \item Let $b_0, b_1\in N\succ M$.  If $C_{b_0}^M\subseteq C_{b_{1}}^M$ then $C_{b_0\meet b_1}^M=C_{b_0}^M$. If none of  $C_{b_0}^M$ and $C_{b_1}^M$ is included in the other, then $b_0\meet b_1\in M$. 
  \item For all $\bla b0,n\in N\succ M$, points of $\dcl(M \bla b0,n)$ are either in $M$ or have the same cut in $M$ as one of the $b_i$.
  \item If $p\in \invtypes_1(\monster)$ then $C_p$ is bounded.
  \end{enumerate}
\end{lemma}
\begin{proof}
The first  part is  clear from the definitions of cut and meet, and the second one follows by  induction. The last part follows from the characterisation of invariant $1$-types. 
\end{proof}

\section{Weak binarity}\label{sec:wbin}
The main result of this section, Theorem~\ref{thm:expdmt}, states that certain expansions of $\tdmt$ are \emph{weakly binary}.  
 It applies for instance to the theory $\mathsf{DTR}$ from~\cite{estevankaplan}, obtained by equipping every set of open cones above a point with a structure elementarily equivalent to the  Random Graph.

Recall that a theory $T$ is  \emph{binary}  iff every formula is equivalent modulo $T$ to a Boolean combination of formulas with at most two free variables. Equivalently, for every set of parameters $B$ and tuples $a,b$,
  \[
    \tp(a/B)\cup \tp(b/B)\cup \tp(ab/\emptyset)\proves \tp(ab/B)
  \]
 Natural examples of such theories are those which eliminate quantifiers in a binary relational language. On the other hand, binary \emph{function} symbols are usually  an obstruction to binarity, as they can be used to write atomic formulas with an arbitrary number of free variables.

This is for instance the case for $\tdmt$, whose language contains the binary function symbol $\meet$: it is easy to see that $\tdmt$  is not binary, nor is any of its expansions by constants.  Even though $\tdmt$ is known to be \emph{ternary} (see~\cite[Corollary~4.6]{simondmt}), this is not sufficient for our purposes: the theory from~\cite[Proposition~2.3]{invbartheory} where $\invtilde$ is not well-defined is ternary as well.

  \begin{defin}\label{defin:wb}
A theory is \emph{weakly binary} iff, whenever $a$, $b$ are tuples such that   $\tp(a/\monster)$ and $\tp(b/\monster)$ are invariant, there is $A\smallsubset \monster$ such that
    \begin{equation}\tag{$\dagger$}\label{eq:wb}
    \tp(a/\monster)\cup \tp(b/\monster)\cup \tp(a,b/A)\proves \tp(a,b/\monster)
  \end{equation}
\end{defin}
Weak binarity was introduced in~\cite{invbartheory} as a sufficient condition for well-definedness of the domination monoid. The class of weakly binary theories clearly contains any theory which happens to have an expansion by constants which is binary. Examples with no binary expansion by constants include the theory of a generic equivalence relation where every equivalence class carries a circular order and, as we will shortly see, $\tdmt$. 

It follows immediately from the definitions that,  in every theory, if $p\in S(\monster)$ is invariant then each of its $1$-subtypes, that is, each of the restrictions of $p$ to one of its variables, is invariant as well.
  It is easily seen, say by using~\cite[Lemma~1.27]{invbartheory} and induction, that if $T$ is weakly binary then the converse holds as well. We record this here for later reference.
\begin{rem}\label{rem:invariantdim1}
Let $T$ be weakly binary and $p\in S(\monster)$. Then $p$ is invariant if and only if every $1$-subtype of $p$ is invariant.
\end{rem}
Before returning to trees, note that this converse is in general false.  For example, in  the theory of Divisible Ordered Abelian Groups, let $p(x_0,x_1)$ be a $2$-type prescribing $x_0,x_1$ to be larger than $\monster$, and such that  both the cofinality of $\set{d\in \monster\mid p\proves x_0-x_1>d}$ and the coinitiality of its complement are not small. Then $p$ is not invariant, even if both of its $1$-subtypes are.
\begin{notation}
 We write $x\parallel y$ to mean that $x\centernot \le y$ and $y\centernot \le x$.
\end{notation}
\begin{lemma}\label{lemma:meetclosedmodulofinite3}
In the theory $\tdmt$, let $b$ be a finite tuple. There is a finite tuple $d$  such that  $\monster bd$ is closed under meets. Moreover, $d$ can be chosen such that additionally, if we let $c\coloneqq\monster\cap d$, then $d\in \dcl(bc)$, and for every $e\in bd\setminus \monster$ such that $C^\monster_e$ is bounded, the following happens.
\begin{enumerate}
\item\label{point:Eae} There is $a_e\in c$ such that $a_e> C^\monster_e$.
\item\label{point:aesamecone} If $C^\monster_e$ has a maximum $g$ and $e$ is in an existing open cone above $g$, then this is the open cone of $a_e$.
\end{enumerate}
\end{lemma}
\begin{proof}
Define a tuple $a$ as follows.  If $C_{b_i}^\monster$ is not bounded, choose $a_i$ to be an arbitrary point of $\monster$ (or, if the reader prefers, leave $a_i$ undefined). If $C_{b_i}^{\monster}$ has a maximum $g$ and $b_i$ is in an open cone above $g$ which intersects $\monster$, let $a_i\in \monster$ be such that $a_i\meet b_i>g$ (see first half of Figure~\ref{figure:choiceofa}); otherwise (second half of the same figure),  choose any $a_i> C_{b_i}^{\monster}$. Note that the closure $\dcl(ba)$ of $ba$ under meets is finite by Remark~\ref{rem:meetfacts}, and  let  $d$ be  a tuple enumerating  $\dcl(ba)\setminus b$. Recall that we defined $c\coloneqq\monster\cap d$, and note that, by construction, $d\in \dcl(bc)$.
\begin{figure}[b]
    \centering
    \begin{minipage}{0.25\linewidth}
      \centering
      \begin{tikzpicture}[scale=0.8, font=\small]
        \node (g) at (0,-1.5) {$\phantom{g}\blacktriangle g$};
        \draw[dotted] (0,-1.5) --(0,-0.5);
                    \draw[dotted] (0,-1.5) edge node [pos=0.6] {$\bullet$} node [pos=0.6, left]{$b_i$} (0,-0.5);
    \draw (0,-0.5)edge node [midway] {$\bullet$} node [midway, right]{$a_i$} (0,0);
    \draw (0,-1.5) edge [bend left=15] (1.5, -0.6);
    \draw (0,0) edge[bend left=15] coordinate[pos=0.3] (a)  (2,1);
    \draw (a) edge[bend left=15] coordinate  [pos=0.3](b) (1.7,1.2);
    \draw (b) edge[bend right=15]  (1.4,1.5);
    \draw (0,0) edge[bend left=15] coordinate[pos=0.6] (d) (1.2,1.7);
    \draw (d) edge[bend right=15]  (0.6,1.7);
    \draw (0,0) edge[bend right=15] coordinate[pos=0.3] (e) (-0.4,1.7);
    \draw (e) edge [bend right=15] coordinate [pos=0.2] (f) coordinate [pos=0.6] (h) (-1.2, 1.5);
    \draw (e) edge [bend left=15] (0.4, 1.5);
    \draw (f) edge [bend right=15] (-2.1, 1.5);
    \draw (h) edge [bend left=15] (-0.8, 1.7);
    \draw (0,0) edge[bend right=15] coordinate[pos=0.3] (i) (-1.8,0.7);
    \draw (i) edge [bend left=15] (-1.2, 1);
    \draw (0, -3) -- (0,-1.5);
    \draw (0,-2.7) edge [bend left=15] coordinate [midway] (j) (1.5,-1.7);
    \draw (j) edge [bend right=15] (1.2, -0.9);
    \draw (0, -2.5) edge [bend right=15] (-1.2, -1);
  \end{tikzpicture}
\end{minipage}
\begin{minipage}{0.25\linewidth}
  \centering
      \begin{tikzpicture}[scale=0.8, font=\small]
        \node (g) at (0,-1.5) {$\phantom{g}\blacktriangle g$};
        \draw[dotted] (0,-1.5) --(0,-0.5);
            \draw [dotted] (0,-1) edge [bend right=15]  node [pos=0.8] {$\bullet$} node [pos=0.8, above]{$b_i$} (-1.5, 0);
    \draw (0,-0.5)edge node [midway] {$\bullet$} node [midway, right]{$a_i$} (0,0);
    \draw (0,-1.5) edge [bend left=15] (1.5, -0.6);
    \draw (0,0) edge[bend left=15] coordinate[pos=0.3] (a)  (2,1);
    \draw (a) edge[bend left=15] coordinate  [pos=0.3](b) (1.7,1.2);
    \draw (b) edge[bend right=15]  (1.4,1.5);
    \draw (0,0) edge[bend left=15] coordinate[pos=0.6] (d) (1.2,1.7);
    \draw (d) edge[bend right=15]  (0.6,1.7);
    \draw (0,0) edge[bend right=15] coordinate[pos=0.3] (e) (-0.4,1.7);
    \draw (e) edge [bend right=15] coordinate [pos=0.2] (f) coordinate [pos=0.6] (h) (-1.2, 1.5);
    \draw (e) edge [bend left=15] (0.4, 1.5);
    \draw (f) edge [bend right=15] (-2.1, 1.5);
    \draw (h) edge [bend left=15] (-0.8, 1.7);
    \draw (0,0) edge[bend right=15] coordinate[pos=0.3] (i) (-1.8,0.7);
    \draw (i) edge [bend left=15] (-1.2, 1);
    \draw (0, -3) -- (0,-1.5);
    \draw (0,-2.7) edge [bend left=15] coordinate [midway] (j) (1.5,-1.7);
    \draw (j) edge [bend right=15] (1.2, -0.9);
    \draw (0, -2.5) edge [bend right=15] (-1.2, -1);
  \end{tikzpicture}
\end{minipage}
\begin{minipage}{0.25\linewidth}
  \centering
  \begin{tikzpicture}[scale=0.8, font=\small]
    \node (g) at (0,-1.5) {$\phantom{g}\blacktriangle g$};
    \draw (0,-0.5)-- (0,0);
            \draw[dotted] (0,-1.5) edge  (0,-0.5);
    \draw (0,0) edge[bend left=15] coordinate[pos=0.3] (a)  (2,1);
    \draw (a) edge[bend left=15] coordinate  [pos=0.3](b) (1.7,1.2);
    \draw (b) edge[bend right=15]  (1.4,1.5);
    \draw (0,0) edge[bend left=15] coordinate[pos=0.6] (d) (1.2,1.7);
    \draw (d) edge[bend right=15]  (0.6,1.7);
    \draw (0,0) edge[bend right=15] coordinate[pos=0.3] (e) (-0.4,1.7);
    \draw (e) edge [bend right=15] coordinate [pos=0.2] (f) coordinate [pos=0.6] (h) (-1.2, 1.5);
    \draw (e) edge [bend left=15] (0.4, 1.5);
    \draw (f) edge [bend right=15] node [pos=0.8] {$\bullet$} node [pos=0.8, below]{$a_i$} (-2.1, 1.5);
    \draw (h) edge [bend left=15] (-0.8, 1.7);
    \draw (0,0) edge[bend right=15] coordinate[pos=0.3] (i) (-1.8,0.7);
    \draw (i) edge [bend left=15] (-1.2, 1);
    \draw (0, -3) -- (0,-1.5);
    \draw (0,-2.7) edge [bend left=15] coordinate [midway] (j) (1.5,-1.7);
    \draw (j) edge [bend right=15] (1.2, -0.9);
    \draw (0, -2.5) edge [bend right=15] (-1.2, -1);
    \draw (0,-1.5) edge [dotted, bend right=15] node [pos=0.6] {$\bullet$} node [pos=0.6, above]{$b_i$} (-1.5, -0.6);
    \draw (0,-1.5) edge [bend left=15] (1.5, -0.6);
  \end{tikzpicture}
\end{minipage}
      \begin{minipage}{0.25\linewidth}
        \centering
      \begin{tikzpicture}[scale=0.8, font=\small]
    \draw[dotted] (0,-1.5) --(0,-0.5);    
    \draw (0,-0.5)-- (0,0);
            \draw[dotted] (0,-1.5) --(0,-0.5);
            \draw [dotted] (0,-1) edge [bend right=15]  node [pos=0.8] {$\bullet$} node [pos=0.8, above]{$b_i$} (-1.5, 0);
    \draw (0,0) edge[bend left=15] coordinate[pos=0.3] (a)  (2,1);
    \draw (a) edge[bend left=15] coordinate  [pos=0.3](b) (1.7,1.2);
    \draw (b) edge[bend right=15]  (1.4,1.5);
    \draw (0,0) edge[bend left=15] coordinate[pos=0.6] (d) (1.2,1.7);
    \draw (d) edge[bend right=15]  (0.6,1.7);
    \draw (0,0) edge[bend right=15] coordinate[pos=0.3] (e) (-0.4,1.7);
    \draw (e) edge [bend right=15] coordinate [pos=0.2] (f) coordinate [pos=0.6] (h) (-1.2, 1.5);
    \draw (e) edge [bend left=15] (0.4, 1.5);
    \draw (f) edge [bend right=15] node [pos=0.8] {$\bullet$} node [pos=0.8, below]{$a_i$} (-2.1, 1.5);
    \draw (h) edge [bend left=15] (-0.8, 1.7);
    \draw (0,0) edge[bend right=15] coordinate[pos=0.3] (i) (-1.8,0.7);
    \draw (i) edge [bend left=15] (-1.2, 1);
    \draw (0, -3) -- (0,-1.5);
    \draw (0,-2.7) edge [bend left=15] coordinate [midway] (j) (1.5,-1.7);
    \draw (j) edge [bend right=15] (1.2, -0.9);
    \draw (0, -2.5) edge [bend right=15]  (-1.2, -1);
  \end{tikzpicture}
\end{minipage}
    \caption{how to choose $a_i$ in the proof of Lemma~\ref{lemma:meetclosedmodulofinite3}. In the first three pictures, $C^\monster_{b_i}$ has a maximum, $g$, denoted by a triangle. In the last picture it does not have one. Solid lines lie in $\monster$, and dotted lines lie in a bigger $\monster_1\satext \monster$.}
  \label{figure:choiceofa}
\end{figure}
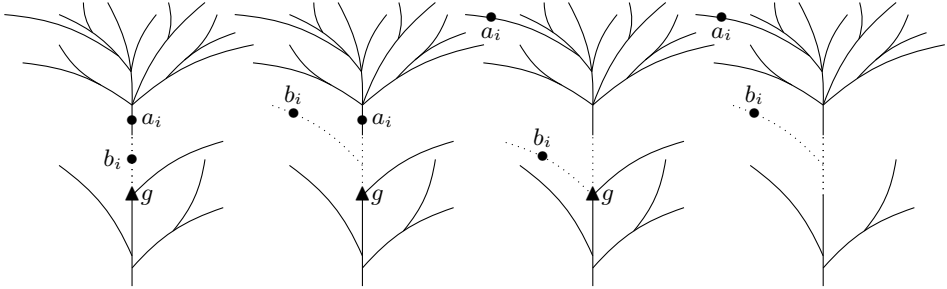

  We now prove the ``moreover'' part, and then show how closure under meets of $\monster bd$ follows. Let $e\in bd\setminus \monster$ have bounded cut. By Lemma~\ref{lemma:cutremarks3}, construction, and the fact that $e\notin \monster$, there is $i<\abs b$ such that $e$ can be written as the meet of $b_i$ with other points (possibly none), either with the same cut as $b_i$, or in $\monster$ and above $C^\monster_{b_i}$ . In particular $e\le b_i$ and  $C^\monster_e=C^\monster_{b_i}$.

\bigcircled{\ref{point:Eae}}    Let $i$ be as above. Since $C^\monster_e=C^\monster_{b_i}$, we have    $a_e\coloneqq a_i> C^\monster_e$.

\bigcircled{\ref{point:aesamecone}} Let $i$ and $a_e$ be as above. By choice of $a_i=a_e$, we have $a_i\meet b_i> g$. By construction and the fact that $e\notin \monster$, we have  $g<e\le b_i$, so $e\meet b_i=e>g$ and $e$ and $b_i$ are in the same open cone above $g$, which is that of $a_i$. This completes the proof of~\ref{point:aesamecone}, hence of the ``moreover'' part.

We are left to prove that $\monster bd$ is closed under meets. As both $\monster$ and $bd$ are, and $\meet$ is commutative, all we need to show is that if  $e\in bd\setminus \monster$  and $f\in \monster$  then $f\meet e\in \monster bd$. If $e$ and $f$ and comparable there is nothing to prove, so assume they are not, i.e.~that $e\parallel f$. It is immediate to notice that if $C_e^\monster$ is unbounded, since $f\in \monster$ and $e\parallel f$, none of  $C_e^\monster$ and $C_f^\monster$ is included in the other. Hence, by the first point of Lemma~\ref{lemma:cutremarks3}, we have $e\meet f\in \monster$. Assume now that $C_e^\monster$ is bounded.
  \begin{claim}
To conclude, it is enough to show that $f\meet e\le f\meet a_e$.
\end{claim}
\begin{claimproof}
 By assumption, commutativity, and idempotency of $\meet$ we have  $f\meet e=(f\meet e)\meet (f\meet a_e)=(f\meet a_e)\meet (a_e\meet e)$.   Since $f\meet a_e$ and $a_e\meet e$ are both predecessors of $a_e$ they are comparable, so their meet is one of them. But $a_e\meet e\in bd$ and $f\meet a_e\in \monster$, so $f\meet e\in \monster bd$.
\end{claimproof}
We prove that $f\meet e\le f\meet a_e$ by cases. Note that, since $f\meet a_e$ and $f\meet e$ are both predecessors of $f$, they are comparable.
  \begin{enumerate}
  \item If $f> C^\monster_e$  then  $C^\monster_{f\meet e}=C^\monster_e$. Suppose additionally that $f\meet a_e>C^\monster_e=C^{\monster}_{f\meet e}$.  Since $f\meet a_e\in \monster$, having $f\meet a_e\le f\meet e$ would contradict $f\meet a_e> C^\monster_{f\meet e}$, and therefore $f\meet e<f\meet a_e$.
  \item If $f> C^\monster_e$  and we are not in the previous case, then $C_e$ has a maximum $g$ and $f\meet a_e=g$, i.e.~$f$ and $a_e$ are in different open cones above $g$. Now,  $e$ can  be either in the same open cone as $a_e$, or in a new one, but in both cases $f\meet e=g=f\meet a_e$.
  \item If $f\centernot > C^\monster_e$ then there is $h\in C^\monster_e$ such that $f\centernot > h$, and then $f\meet h=f\meet (h\meet e)=f\meet e$. As $a_e> C^\monster_e$ in particular $a_e> h$, hence by definition of meet we must have $f\meet a_e=f\meet h=f\meet e$.\qedhere
  \end{enumerate}
\end{proof}
\begin{defin}\label{defin:binconexp}
  A \emph{binary cone-expansion} of $\tdmt$ is a theory $T$ in a language $L=\lmt\cup \set{R_j, P_{j'}\mid j\in J, j'\in J'}$ satisfying the following properties.
  \begin{enumerate}
  \item  Every $P_{j'}$ is a unary relation symbol; every $R_j$ is a binary relation symbol.
  \item $T$ is a completion of $\tdmt$ and eliminates quantifiers in $L$.
  \item Every $R_j$ is \emph{on open cones}, in the sense that
    \begin{enumerate}
    \item $R_j(x,y)\implica x\parallel y$, and 
    \item if $x\parallel y$ and
      $x',y'$ are such that $x\meet x'>x\meet y$ and
      $y\meet y'> x\meet y$ then $R_j(x,y)\coimplica R_j(x',y')$.
    \end{enumerate}
  \end{enumerate}
  If additionally $J'=\emptyset$ we say that  $T$ is  a \emph{purely binary cone-expansion} of $\tdmt$.
\end{defin}
\begin{eg}
A purely binary cone-expansion of $\tdmt$ is  $\mathsf{DTR}$,  axiomatised by taking $J=\set R$, $J'=\emptyset$, and saying that, for all $g$, the structure induced by $R$  on the (imaginary sort of) open cones above $g$ is  elementarily equivalent to the Random~Graph. See~\cite{estevankaplan} for  $\mathsf{DTR}$, and for a more general analysis of theories of trees with relations on sets of open cones.
\end{eg}
\begin{eg}
  Another theory examined in~\cite{estevankaplan}, called $\mathsf{DTE}_2$,  is defined similarly to $\mathsf{DTR}$, but instead of the  Random Graph it uses the Fra\"iss\'e limit of all finite structures with two equivalence relations. More generally, one can define $\mathsf{DTE}_n$ in an analogous fashion. The results of this paper apply to these theories as well even if, strictly speaking, they do not satisfy Definition~\ref{defin:binconexp}, since the latter requires the $R_j$ to be irreflexive.  This can easily be circumvented by observing that, if $E$ is an equivalence relation and $\Delta$ is the diagonal, then $E$ and $E\setminus\Delta$ are interdefinable.
\end{eg}

\begin{thm}\label{thm:expdmt}
  Every binary cone-expansion of $\tdmt$ 
  is weakly binary.
\end{thm}
\begin{proof}
  Let $b^0,b^1$ be tuples each having invariant global type.    By quantifier elimination it is enough to find a finite tuple $c\in \monster$ such that $\tp_x(b^0/\monster)\cup \tp_y(b^1/\monster)\cup \tp_{xy}(b^0b^1/c)$ decides all the atomic relations in $L$ between points of $b^0$, $b^1$, $\monster$, and their meets. Apply Lemma~\ref{lemma:meetclosedmodulofinite3}  to  $b\coloneqq b^0b^1$,  let $d$ be the resulting tuple and set $c\coloneqq d\cap \monster$. We want to show that 
  \[
\pi\coloneqq    \tp(b^0/\monster)\cup \tp(b^1/\monster)\cup \tp(b/c)\proves \tp(b/\monster)
  \]
If $e$ and $f$ are both in $bd$ then $e,f\in \dcl^{\lmt}(bc)$, hence $\tp(b/c)$ entails $\tp(ef/\emptyset)$, and in particular decides all formulas of the forms $R_j(e,f)$ and $P_{j'}(e)$.
\begin{claim}
We have   $\pi\proves \tp^{\lmt}(b/(\monster\restr \lmt))$.
\end{claim}
\begin{claimproof}
Since $\monster bd$ is closed under meets  we  only need to show that the position of all the $e\in d\setminus \monster b$ with respect to $\monster$ is determined. By Lemma~\ref{lemma:cutremarks3} and the fact that $e\in \dcl^{\lmt}(bc)\setminus \monster b$ there is  $i<\abs b$ such that $e< b_i$ and  $C^\monster_e=C^\monster_{b_i}$; note that this information is deduced by $\pi$, because $e$ is a meet of points in $bc$. If $C^\monster_e$ is unbounded, we are done. Otherwise, if $a_e\in c$ is as in Lemma~\ref{lemma:meetclosedmodulofinite3}, all we need to decide is whether $e$ is below or incomparable with $\set{h\in \monster\mid h>a_e\meet e}$. This is decided by whether $a_e> e$ or not, and this information is in $\tp(b/c)$.
\end{claimproof}
We then need to take care of formulas of the form $R_j(e,f)$ for $e\in d\setminus\monster b$ and $f\in \monster$; the argument for formulas of the form $R_j(f,e)$ is identical \emph{mutatis mutandis}. If $e\le f$ or $f\le e$, by hypothesis we must have $\neg R_j(e,f)$, so we may assume that $e\parallel f$.  We distinguish three cases; the fact that, by the  Claim, $\pi$ implies the position of $e$ with respect to $\monster$ will be used tacitly.
\begin{enumerate}
\item\label{point:abovecone} Assume first $e\meet f> C^\monster_e$. Some subcases of this case are depicted in Figure~\ref{figure:conexplemma}. By assumption $C^\monster_e$ is bounded and, if $a_e\in c$ is as in Lemma~\ref{lemma:meetclosedmodulofinite3}, we have $a_e\meet f> C^\monster_e=C^\monster_{e\meet f}$. Since $a_e\meet f$ and $e\meet f$ must be comparable, and $a_e\meet f\in \monster$, this implies $a_e\meet f> e\meet f$, so  $a_e$ and $f$ are in the same open cone above $e\meet f$. By our hypotheses on $T$ then $R_j(e,f)\coimplica R_j(e,a_e)$, but $a_e\in c$ and $e\in \dcl^{\lmt} (bc)$, so since $\pi\proves\tp(b/c)$ we are done.
  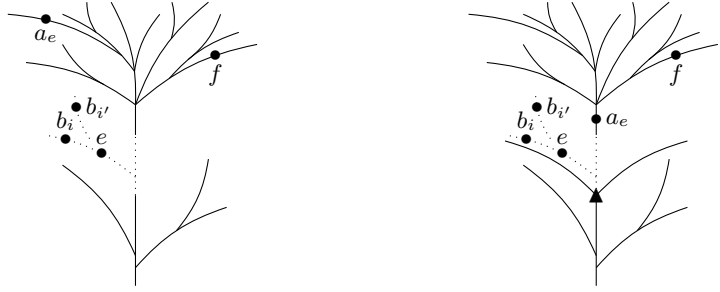
\begin{figure}
    \centering
    \begin{minipage}{0.5\linewidth}
      \centering
      \begin{tikzpicture}[scale=0.8, font=\small]
    \draw[dotted] (0,-1.5) --(0,-0.5);    
    \draw (0,-0.5)-- (0,0);
    \draw [dotted] (0,-1.2) edge [bend right=15] node [pos=0.4] (the_e) {$\bullet$} node [pos=0.4, above]{$e$} node [pos=0.8] {$\bullet$} node [pos=0.8, above]{$b_i$} (-1.5, -0.5);
    \draw [dotted] (the_e) edge [bend left=15] node [pos=0.7] {$\bullet$} node [pos=0.7, right]{$b_{i'}$}(-1, 0.2);
    \draw (0,0) edge[bend left=15] coordinate[pos=0.3] (a) node [pos=0.7] {$\bullet$} node [pos=0.7, below]{$f$} (2,1);
    \draw (a) edge[bend left=15] coordinate  [pos=0.3](b) (1.7,1.2);
    \draw (b) edge[bend right=15]  (1.4,1.5);
    \draw (0,0) edge[bend left=15] coordinate[pos=0.6] (d) (1.2,1.7);
    \draw (d) edge[bend right=15]  (0.6,1.7);
    \draw (0,0) edge[bend right=15] coordinate[pos=0.3] (e) (-0.4,1.7);
    \draw (e) edge [bend right=15] coordinate [pos=0.2] (f) coordinate [pos=0.6] (h) (-1.2, 1.5);
    \draw (e) edge [bend left=15] (0.4, 1.5);
    \draw (f) edge [bend right=15] node [pos=0.7] {$\bullet$} node [pos=0.7, below]{$a_e$} (-2.1, 1.5);
    \draw (h) edge [bend left=15] (-0.8, 1.7);
    \draw (0,0) edge[bend right=15] coordinate[pos=0.3] (i) (-1.8,0.7);
    \draw (i) edge [bend left=15] (-1.2, 1);
    \draw (0, -3) -- (0,-1.5);
    \draw (0,-2.7) edge [bend left=15] coordinate [midway] (j) (1.5,-1.7);
    \draw (j) edge [bend right=15] (1.2, -0.9);
    \draw (0, -2.5) edge [bend right=15] (-1.2, -1);
  \end{tikzpicture}
\end{minipage}
\begin{minipage}{0.5\linewidth}
        \centering
      \begin{tikzpicture}[scale=0.8, font=\small]
        \node (g) at (0,-1.5) {$\blacktriangle$};
        \draw[dotted] (0,-1.5) --(0,-0.5);
            \draw [dotted] (0,-1.2) edge [bend right=15] node [pos=0.4] (the_e) {$\bullet$} node [pos=0.4, above]{$e$} node [pos=0.8] {$\bullet$} node [pos=0.8, above]{$b_i$} (-1.5, -0.5);
    \draw [dotted] (the_e) edge [bend left=15] node [pos=0.7] {$\bullet$} node [pos=0.7, right]{$b_{i'}$}(-1, 0.2);
    \draw (0,-0.5)edge node [midway] {$\bullet$} node [midway, right]{$a_e$} (0,0);
    \draw (0,-1.5) edge [bend left=15] (1.5, -0.6);
    \draw (0,-1.5) edge [bend right=15] (-1.5, -0.6);
    \draw (0,0) edge[bend left=15] coordinate[pos=0.3] (a) node [pos=0.7] {$\bullet$} node [pos=0.7, below]{$f$} (2,1);
    \draw (a) edge[bend left=15] coordinate  [pos=0.3](b) (1.7,1.2);
    \draw (b) edge[bend right=15]  (1.4,1.5);
    \draw (0,0) edge[bend left=15] coordinate[pos=0.6] (d) (1.2,1.7);
    \draw (d) edge[bend right=15]  (0.6,1.7);
    \draw (0,0) edge[bend right=15] coordinate[pos=0.3] (e) (-0.4,1.7);
    \draw (e) edge [bend right=15] coordinate [pos=0.2] (f) coordinate [pos=0.6] (h) (-1.2, 1.5);
    \draw (e) edge [bend left=15] (0.4, 1.5);
    \draw (f) edge [bend right=15] (-2.1, 1.5);
    \draw (h) edge [bend left=15] (-0.8, 1.7);
    \draw (0,0) edge[bend right=15] coordinate[pos=0.3] (i) (-1.8,0.7);
    \draw (i) edge [bend left=15] (-1.2, 1);
    \draw (0, -3) -- (0,-1.5);
    \draw (0,-2.7) edge [bend left=15] coordinate [midway] (j) (1.5,-1.7);
    \draw (j) edge [bend right=15] (1.2, -0.9);
    \draw (0, -2.5) edge [bend right=15] (-1.2, -1);
  \end{tikzpicture}
\end{minipage}
    \caption{two subcases of case~\ref{point:abovecone} in the proof of Theorem~\ref{thm:expdmt}, where $e\meet f> C^\monster_e$. In the first picture, $C^\monster_e$ does not have a maximum. In the second picture it has one, denoted by a triangle.  Solid lines lie in $\monster$, and dotted lines lie in a bigger $\monster_1\satext \monster$. Other subcases are similar, and correspond to different arrangements of $a_e$ and $f$, e.g.~$a_e>f$.}
  \label{figure:conexplemma}
\end{figure}
\item Assume now that $e\meet f\centernot > C^\monster_e$ and there is $h\in \monster$ such that $e\meet h>e\meet f$.
  Then $e$ is in the same open cone above $e\meet f$ as $h$, hence $R_j(e,f)\coimplica R_j(h,f)$. Since $f,h\in \monster$ we are done.
\item If $e\meet f\centernot > C_e^\monster$ but there is no $h$ as in the previous point, then $C^\monster_e$ must have a maximum $g$, which needs to equal $e\meet f$,  and since $e\parallel f$ we need to have $f>g$. If $e$ is in an existing open cone above $g$, since the $R_j$ are on open cones, we are done, so assume it is in a new one.  Since $e\in \dcl^{\lmt}(bc)$, by Lemma~\ref{lemma:cutremarks3} this can only happen if there is $i<\abs b$ such that  $e\le b_i$, hence $e$ shares the same open cone above $g$ as $b_i$. Again, since the $R_j$ are on open cones,  we are done. \qedhere
\end{enumerate}
\end{proof}
\begin{rem}
Weak binarity was introduced as a sufficient condition for $\invtilde$ to be well-defined. Consequently, in Definition~\ref{defin:wb}, we only require~\eqref{eq:wb} to hold for tuples $a$, $b$ such that $\tp(a/\monster)$ and $\tp(b/\monster)$ are invariant. Nevertheless, in the proofs above we never used invariance, hence binary cone-expansions of $\tdmt$ satisfy a condition slightly stronger than weak binarity, obtained  by requiring~\eqref{eq:wb} from Definition~\ref{defin:wb}  to hold for \emph{all} tuples $a$, $b$.
\end{rem}
\section{The domination monoid: pure trees}\label{sec:invtildedmt}
We now compute the domination monoid in $\tdmt$. We first recall  briefly its definition and some of its basic properties for the reader's convenience, and otherwise refer to~\cite{invbartheory}. See also~\cite{mythesis} for a more extensive treatment.

It is well-known that, if $A\smallsubset  \monster\subseteq B$ and  $p\in \invtypes_x(\monster, A)$, then there is a unique $p\invext B$ extending $p$ to an $A$-invariant type over $B$,  given by requiring, for each  $\phi(x;y)\in L(A)$  and $b\in B$, 
 \[
 \phi(x;b)\in p\invext B\iff \tp(b/A)\in (d_p\phi(x;y))(y)
\]
This canonical extension to bigger parameter sets allows to define the \emph{tensor product} of $p\in \invtypes_x(\monster,A)$ with any  $q\in S_y(\monster)$ as follows.   Fix $b\models q$; for each $\phi(x,y)\in L(\monster)$, define
 \[
 \phi(x,y)\in p(x)\otimes q(y)\iff \phi(x,b)\in p\invext \monster b
 \]

Some authors denote by $q(y)\otimes p(x)$ what we denote by $p(x)\otimes q(y)$. 

It is an easy exercise to show that the product $\otimes$ does not depend on $b\models q$, nor on the choice of a base of invariance for $p$, that it is associative, and that if $p$, $q$ are both $A$-invariant then so is $p\otimes q$.

In what follows,  when considering types $p(x)$, $q(y)$, say, we assume without loss of generality that the tuples of variables $x$ and $y$ are disjoint. 
\begin{defin} \label{defin:domination}
Let $p\in S_x(\monster)$ and $q\in S_y(\monster)$. We say that $p$ \emph{dominates} $q$,  and write $p\doms q$, iff there are some small $A$ and some $r\in S_{xy}(A)$ such that
\begin{itemize}
\item $r\in S_{pq}(A)\coloneqq\set{r\in S_{xy}(A)\mid r\supseteq (p\restr A)\cup (q\restr A)}$, and
\item $p(x)\cup r(x,y)\proves q(y)$.
\end{itemize}
In this case, we say that $r$ is a \emph{witness} to, or \emph{witnesses} $p\doms q$. We say that $p$ and $q$ are \emph{domination-equivalent}, and write $p\domeq q$, iff $p\doms q$ and $q\doms p$.
  \end{defin}

  \begin{eg}\label{eg:pushf}
Suppose that $q(y)$ is the \emph{pushforward} of $p(x)$ under the $A$-definable function $f$, namely $q(y)\coloneqq\set{\phi(y)\mid p(x)\proves \phi(f(x))}$. In this case, and in the more general one where $\abs y >1$ and $f$ is a tuple of definable functions, we have  $p\doms q$, witnessed by any completion of $(p(x)\restr A)\cup (q(y)\restr A)\cup \set{y=f(x)}$.
\end{eg}
In Definition~\ref{defin:domination} we are not requiring $p\cup r$ to be a complete global type in variables $xy$; in other words, domination is ``small-type semi-isolation'', as opposed to ``small-type isolation'', i.e.~$\mathrm{F}^\mathrm{s}_{\kappa}$-isolation in the notation of~\cite[Chapter~IV]{classificationtheoryshelah}. While it is easy to see that $\mathrm{F}^\mathrm{s}_{\kappa}$-isolation is the same as domination in every weakly binary theory, the two relations are in general distinct. This can be seen in the theory below; the reader who dislikes random digraphs may feel free to replace them with generic equivalence relations.

\begin{eg}\label{eg:semineeded}
Work in a $2$-sorted language, with sorts $O$ (``objects'') and $D$ (``digraphs'').  Let $L\coloneqq\set{\pow E{O^2}, \pow P{O}, \pow R{O^2\times D}}$, a relational language with arities  indicated as superscripts. Consider the following universal axioms.
  \begin{enumerate}
  \item $E$ is an equivalence relation.
  \item $R(x,y,w)\implica E(x,y)$.
  \item $R(x,y,w)\implica\neg R(y,x,w)$.
  \end{enumerate}
The finite structures satisfying these axioms form a Fra\"iss\'e class; let $T$  be the theory of its limit. In a model of $T$, the sort $O$ carries an equivalence relation with infinitely many classes.  On each class $a/E$ the predicate $P$ is infinite and coinfinite, and each point of $D$ induces a random digraph on each $a/E$. Different random digraphs, on the same $a/E$ or on different ones, interact generically with $P$ and with each other, but no digraph has an edge across different classes. Let $x$ be a variable of sort $O$, define  $\pi(x)\coloneqq \set{\neg E(x,d)\mid d\in \monster}$,  and   let $p(x)\coloneqq \pi(x)\cup \set{P(x)}$ and  $q(y)\coloneqq \pi(y)\cup \set{\neg P(y)}$. By quantifier elimination and the lack of edges across different classes,  $p$ and $q$ are complete global types, in fact $\emptyset$-invariant ones. Any $r\in S_{pq}(\emptyset)$ containing   $\rho(x,y)\coloneqq E(x,y)\land P(x)\land \neg P(y)$ witnesses simultaneously that $p\doms q$ and that $q\doms p$, since $p\cup \set \rho\proves q$ and $q\cup \set\rho\proves p$. Note that the predicate $P$ forbids $r$ from containing $x=y$.  By genericity, there is no small $A$ such that for some $r\in S_{pq}(A)$ the partial type $p\cup r$ decides, for all $d\in \monster$, whether $R(x,y,d)$ holds, and similarly for $q\cup r$. Hence, for all $a\models p$ and $b\models q$, neither  $\tp(a/\monster b)$ nor $\tp(b/\monster a)$ is   $\mathrm{F}^\mathrm{s}_{\kappa}$-isolated.
\end{eg}
It can be shown that $\domeq$ is a preorder, hence $\domeq$ is an equivalence relation.
Let $\invtilde$ be the quotient of $\invtypes(\monster)$ by $\domeq$. The partial order induced by $\doms$ on $\invtilde$ will, with abuse of notation, still be denoted by $\doms$, and we call $(\invtilde, \doms)$ the \emph{domination poset}. This poset has a minimum, the (unique) class of \emph{realised types}, i.e.~global types realised in $\monster$, denoted by $\class 0$.

If $T$ is  such that  $(\invtypes(\monster), \otimes, \doms)$ is a preordered semigroup, we say that $\otimes$ \emph{respects} $\doms$. In particular, then $\domeq$ is a congruence with respect to $\otimes$, and induces a well-defined operation on $\invtilde$, still denoted by $\otimes$, easily seen to have neutral element $\class 0$. Call the structure $(\invtilde, \otimes, \class 0, \doms)$ the \emph{domination monoid}. We usually denote it simply by $\invtilde$, and say that \emph{$\invtilde$ is well-defined} to mean that $\otimes$ respects $\doms$; this should cause no confusion since  $\invtilde$ is \emph{always} well-defined as a poset.

As shown in~\cite{invbartheory}, $\invtilde$  need not be well-defined in general, but it is in certain classes of theories, such as stable ones. More relevantly to the present endeavour, we recall the following.
\begin{fact}[\!\!{\cite[Corollary~1.30]{invbartheory}}]\label{fact:wbinwd}
In every weakly binary theory, the partially ordered monoid $(\invtilde, \otimes, \class 0, \doms)$ is well-defined.
\end{fact}

Recall that two types $p(x),q(y)\in S(B)$ are \emph{weakly orthogonal}, denoted by $p\wort q$, iff $p(x)\cup q(y)$ is a complete type in $S_{xy}(B)$. In particular, if $p,q\in\invtypes(\monster)$, then $p(x)\otimes q(y)=q(y)\otimes p(x)$, since both extend $p(x)\cup q(y)$. We will also need the following two facts.

\begin{fact}[\!\!{\cite[Corollary~4.7]{invnip}}]\label{fact:wortnipinv} Let $T$ be $\mathsf{NIP}$ and  $\set{p_i\mid i\in I}$ be a family of types $p_i(x^i)\in \invtypes(\monster)$ such that if $i\ne j$ then $p_i\wort p_j$. Then $\bigcup_{i\in I} p_i(x^i)$ is complete. 
\end{fact}

\begin{fact}[\!\!{\cite[Proposition~3.13 and Corollary~3.14]{invbartheory}}]
Let $p_0, p_1\in \invtypes(\monster)$, $q\in S(\monster)$, and assume that $p_0\doms p_1$. If $p_0\wort q$, then $p_1\wort q$. If $p_0\doms q$ and $p_0\wort q$, then $q$ is realised.
\end{fact}
In particular we may endow $\invtilde$ with an additional relation, induced by $\wort$ and denoted by the same symbol.

In what follows, if $r\in S_{pq}(A)$  witnesses $p\doms q$,  by passing to a suitable extension of $r$ there is no harm in enlarging $A$, provided it stays small, which we may do tacitly; if $p,q$ are invariant, we will furthermore assume  $A$ to be large enough so that $p,q\in \invtypes(\monster, A)$. Sometimes, we say that $r$ witnesses domination even if it is not complete, but merely consistent with $(p(x)\cup q(y))\restr A$. In that case, we mean that any of its completions to a type in $S_{pq}(A)$ does. Similarly, we sometimes just write  e.g.~``put in $r$ the formula $\phi(x,y)$''.

\begin{pr}\label{pr:treebasic}
The following statements hold in $\tdmt$.
  \begin{enumerate}
  \item Suppose all coordinates of $p\in S(\monster)$ have the same cut $C_{0}$,  all coordinates of $q\in S(\monster)$ have the same cut $C_1$, and $C_0\ne C_1$. Then $p\wort q$.
  \item Let $C$ be a cut with maximum $g$.  Suppose that all $1$-subtypes of $p$ are of kind (Ib) with cut $C$ and all $1$-subtypes of $q$ are  of kind (II) with cut $C$, or that all $1$-subtypes of $p,q$ are of kind (Ib) with cut $C$, but no open cone above $g$ contains both a coordinate of $p$ and one of $q$.  Then $p\wort q$.
  \item Every $1$-type of kind (IIIa) is domination-equivalent to the unique $1$-type of kind (Ia) with the same cut. Every $1$-type of  kind (IIIb) is domination-equivalent to the unique  $1$-type of kind (Ib) with the same cut and, if this cut has a maximum $g$, the same open cone above $g$.
  \end{enumerate}
  In particular, if $p,q\in \invtypes_1(\monster)$, then either $p\wort q$ or $p\domeq q$.
\end{pr}
\begin{proof}\*
  \bigcircled 1 By quantifier elimination and the first two points of Lemma~\ref{lemma:cutremarks3}.

  \bigcircled 2 This does not follow from the previous point because such types have the same cut, but it is still easy from quantifier elimination and the fact that the open cones in which types of kind (II) concentrate are new, while those of types of kind (Ib) are realised.

  \bigcircled 3 We give a proof for kind (IIIa) which may be easily modified to yield one for kind (IIIb). Suppose that $c\in \monster$ and $A\smallsubset \monster$ are such that  \[p(x)\proves\set{x\centernot \le b\mid b\in \monster}\cup\set{x\meet c<a\mid a\in A}\cup\set{x\meet c> b\mid b\in \monster, b< A}\]
    Let  $q$ be the pushforward of $p$ under the definable function $x\mapsto x\meet c$. By this very description  $p(x)\doms q(y)$ (cf.~Example~\ref{eg:pushf}) and, by definition of kind (IIIa),  $q$ is of kind (Ia), and clearly has the same cut as $p$. To prove $q(y)\doms p(x)$, use some $r\in S_{pq}(A)$ extending $\set{a>(x\meet c)>y\mid a\in A}$; since $r$ contains $p(x)\restr A$, which proves $x\centernot \le a$ for all $a\in A$, we are done. See Figure~\ref{figure:somedomination}.
  \end{proof}
  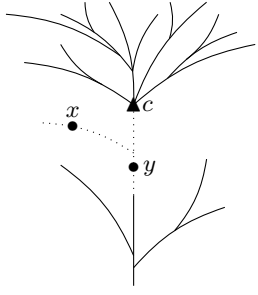
\begin{figure}[t]
  \centering
  \begin{tikzpicture}[scale=0.8, font=\small]
    \node (g) at (0,0) {$\phantom c\blacktriangle c$};
    \draw[dotted] (0,-1.5) edge node [pos=0.3](x) {$\bullet$} node [pos=0.3, right]{$y$} (0,0);    
    \draw [dotted] (0,-0.8) edge [bend right=15] node [pos=0.7] {$\bullet$} node [pos=0.7, above]{$x$} (-1.5, -0.3);
    \draw (0,0) edge[bend left=15] coordinate[pos=0.3] (a) (2,1);
    \draw (a) edge[bend left=15] coordinate  [pos=0.3](b) (1.7,1.2);
    \draw (b) edge[bend right=15]  (1.4,1.5);
    \draw (0,0) edge[bend left=15] coordinate[pos=0.6] (d) (1.2,1.7);
    \draw (d) edge[bend right=15]  (0.6,1.7);
    \draw (0,0) edge[bend right=15] coordinate[pos=0.3] (e) (-0.4,1.7);
    \draw (e) edge [bend right=15] coordinate [pos=0.2] (f) coordinate [pos=0.6] (h) (-1.2, 1.5);
    \draw (e) edge [bend left=15] (0.4, 1.5);
    \draw (f) edge [bend right=15] (-2.1, 1.5);
    \draw (h) edge [bend left=15] (-0.8, 1.7);
    \draw (0,0) edge[bend right=15] coordinate[pos=0.3] (i) (-1.8,0.7);
    \draw (i) edge [bend left=15] (-1.2, 1);
    \draw (0, -3) -- (0,-1.5);
    \draw (0,-2.7) edge [bend left=15] coordinate [midway] (j) (1.5,-1.7);
    \draw (j) edge [bend right=15] (1.2, -0.9);
    \draw (0, -2.5) edge [bend right=15] (-1.2, -1);
  \end{tikzpicture}
  \caption{proof of Proposition~\ref{pr:treebasic}, how to show that $q(y)\doms p(x)$. In this picture $A$ only contains the point $c$, denoted by a triangle. Solid lines lie in $\monster$, and dotted lines lie in a bigger $\monster_1\satext \monster$.}
  \label{figure:somedomination}
\end{figure}

  In Proposition~\ref{pr:treebasic}, it is important to work with $\domeq$, as opposed to the finer relation $\equidom$ of \emph{equidominance}, obtained by requiring that domination of $q$ by $p$ and of $p$ by $q$ can be witnessed by the same $r$. While using some $r$  containing\footnote{Here,  the domain $A$ of $r$ has to be large enough for $p,q$ to be $A$-invariant. Using $q(y)\cup\set{x\meet c=y}$ alone is not  enough to show $x\ne c$, and if $\set{a\in \monster \mid p\proves x\meet c<a}$ does not have a minimum then no single formula is enough to show $q\doms p$.} $x\meet c=y$ would still work to show that every type of kind (IIIa) is  equidominant to one of type (Ia),  this would not work for kind (IIIb), as shown below.
\begin{rem}
  Let $p(x)$ and $q(y)$ be the types respectively of kind (IIIb) and (Ib) with cut $\emptyset$. Then $p\nequidom q$.
\end{rem}
\begin{proof}
  Suppose that $r(x,y)$ witnesses equidominance. If $r(x,y)\proves x\meet y<y$, then $p(x)\cup r(x,y)\centernot \proves q(y)$, since by quantifier elimination and compactness it cannot prove all formulas $y<d$, for $d\in \monster$. If $r(x,y)\proves x\meet y=y$, then $q(y)\cup r(x,y)\centernot\proves p(x)$, since it cannot prove all formulas $x\centernot\le d$.
\end{proof}
\begin{pr}\label{pr:tdmt_types_behaviour}
  In the theory of dense meet-trees the following hold.
  \begin{enumerate}
    \item Types of kind (Ia) and (Ib) are idempotent modulo equidominance. 
  \item If $p$ is of kind (II) and $m<n\in \omega$ then $\pow pm\ndoms \pow pn$.
  \end{enumerate}
\end{pr}
\begin{proof}\*
  \bigcircled 1   Let $A$ be such that $p$ is $A$-invariant. It follows easily from quantifier elimination that in order to show $p(x_1)\otimes p(x_0)\equidom p(y)$ it is enough to put in $r\in S_{pq}(A)$  the formula $x_0=y$.

  \bigcircled 2 For notational simplicity we show the case $m=1$, $n=2$, the general case being analogous. Suppose that $p$ is the type of an element in a new open cone above $g$, i.e.~$p(y)\proves \set{y>g}\cup\set{y\meet b=g\mid b\in \monster, b>g}$. We want to show that there is no small $r(y,x_0,x_1)$ such that  $p(y)\cup r\proves p(x_1)\otimes p(x_0)$. Since $\pow p2\restr \set{g}$ proves $x_0\meet x_1=g$, i.e.~that the cones of $x_0$ and $x_1$ are distinct, there is $i<2$  such that $r\proves y\meet x_i=g$. Since $r$ is small there is $d>g$ in $\monster$ such that  $p(y)\cup r\centernot \proves x_i\meet d= g$; in other words it is not possible, with a small type, to say that $x_i$ is in a new open cone, unless it is the same cone as $y$, but $y$ cannot be in the open cones of $x_0$ and $x_1$ simultaneously.
\end{proof}

Since by Theorem~\ref{thm:expdmt}   dense meet-trees are weakly binary,  $\invtilde$ is well-defined by Fact~\ref{fact:wbinwd}.  By the results above, (domination-equivalence classes of) $1$-types of kind (II) generate a copy of $\mathbb N$, while all other (classes of) $1$-types are idempotent.  We have also seen that if $p,q$ are nonrealised $1$-types, then either $p\wort q$ or $p\domeq q$. In particular, all pairs of $1$-types commute modulo domination-equivalence. To complete our study we need one last ingredient.
\begin{pr}\label{pr:tdmtgen}
  In $\tdmt$, every invariant type is domination-equivalent to a product of invariant $1$-types.
\end{pr}
\begin{proof}
By  Fact~\ref{fact:wortnipinv} and Proposition~\ref{pr:treebasic} we  reduce to showing the conclusion for types $p(x)$ consisting of elements all with the same cut $C_p$.

  Assume first that $C_p$ does not have a maximum, let $c\models p(x)$ and let $d\in \monster$  be such that $d>C_p$, which exists by the last point of Lemma~\ref{lemma:cutremarks3}. Let $H=\set{h_0(c),\ldots, h_n(c)}$ be the (finite, by Remark~\ref{rem:meetfacts})  set of points in $\dcl(cd)$ such that  $d>h_i(c)$, where each $h_i(x)$ is a $\set d$-definable function. By semilinearity $H$ is linearly ordered; suppose, up to reindexing, that $h_0(c)=\min H$ and $h_n(c)=\max H$. We have two subcases. If $C_p$ has small cofinality, let $q(y)$ be of kind (Ib) with $C_q=C_p$. Let $A$ be such that  $p,q\in \invtypes(\monster, A)$, let $r(x,y)\in S_{pq}(Ad)$ contain the formula $h_n(x)<y$, and note that $q(y)\cup r(x,y)$ implies the type over $\monster$ of each point of $\dcl(xd)$, i.e.~of the closure of $xd$ under meets. It follows from quantifier elimination that $q\cup r\proves p$. To prove  $p\cup r\proves q$, use instead some   $r$ containing the formula $y<h_0(x)$. In the other subcase, $\set{e\in \monster\mid C_p<e<d}$ has small coinitiality. The argument is  analogous, except we  use an $r$ containing $h_0(x)>y$ to show $q\cup r\proves p$ and one containing $h_n(x)<y$ to show $p\cup r\proves q$.

  Suppose now that $C_p$ has maximum $g$. Assume without loss of generality that $\bla c0,{k-1}$ are the points of $c$ such that there is $d_i\in \monster$ with  $d_i \meet c_i>g$. In other words, these are the points in  existing open cones above $g$, and $\bla ck,{\abs c-1}$ are in new open cones. Again by quantifier elimination, we have $\tp(\bla c0,{k-1}/\monster)\wort \tp(\bla ck,{\abs c-1}/\monster)$, so we can deal with the two subtypes separately. Similarly, by using weak orthogonality we may split $c_{<k}$ further, and we may assume that for $i<\ell$, say, all $c_i$ are in the same open cone, say that of the point $d\in \monster$. It is now enough to proceed as in the previous case, by taking $q(y)$ to be the type of kind (Ib) with the same cut and open cone above $g$. As for $\bla ck,{\abs c-1}$, let $H$ be the set of minimal elements of $\dcl(\bla ck,{\abs c-1})\setminus \monster$. Let $q(y)$ be the type of kind (II) above $g$. To conclude, let $r$ identify elements of $H$ with coordinates of a realisation of $\pow q{\abs H}$.
\end{proof}
The previous results yield the following characterisation of $\invtilde$ in $\tdmt$.
\begin{thm}\label{thm:dmt}
  In dense meet-trees, $\invtilde\cong\pfin(X)\oplus\bigoplus_{g\in \monster} \mathbb N$. 
 Generators of copies of $\mathbb N$ correspond to types of elements in a new open cone above a point $g\in \monster$, i.e.~to types of kind (II), while each point of $X$ corresponds to, either:
\begin{enumerate}
\item a linearly ordered subset of $\monster$ with small coinitiality, modulo mutual coinitiality; this corresponds to types of kind (Ia)/(IIIa);
\item a cut with no maximum, but with small cofinality; this  corresponds to some types of kind (Ib)/(IIIb);
\item an existing open cone above an existing point; this corresponds to the rest of the types of kind (Ib)/(IIIb).
\end{enumerate}
\end{thm}

\section{The domination monoid: expansions}\label{sec:invtildeexp}
In this section we  generalise Theorem~\ref{thm:dmt} to purely binary cone-expansions of $\tdmt$, such as $\mathsf{DTR}$, by replacing the direct summands isomorphic to $\mathbb N$ with the domination monoids of the structures induced on sets of open cones. In  $\tdmt$ these are pure sets, which are easily seen to  have domination monoid isomorphic to $\mathbb N$ and generated by the $\domeq$-class of the unique nonrealised $1$-type.

Before restricting our attention to purely binary cone-expansions, we observe a phenomenon which can arise in the presence of unary predicates.
Suppose  for instance that $L=\lmt\cup \set P$, where $P$ is a unary predicate symbol interpreted as a branch of $\monster$, i.e.~a maximal linearly ordered subset. In this case, there is an $\emptyset$-invariant type $p$ with cut $C_p=P(\monster)$, and by the last point of Lemma~\ref{lemma:cutremarks3} $p\restr\lmt$ is not invariant. Another binary cone-expansion of $\tdmt$ where there is an invariant type $p$ such that $p\restr \lmt$ is not invariant can be obtained by taking as $P(\monster)$ a bounded linearly ordered subset with no supremum.  However, using unary predicates is the \emph{only} way to obtain such behaviour in a binary cone-expansion of $\tdmt$, as we are about to show. We refer the reader interested to preservation of invariance under reducts to~\cite{rideausimon}.

Denote by $G_g$ the \emph{closed cone} above $g$, namely $\set{b\in \monster \mid b\ge g}$.
\begin{defin}\label{defin:tamefla}
Let $T$ be an expansion of $\tdmt$. We call a formula  $\phi(x)$ with $\abs x=1$ \emph{tame} iff it has the following property: there is a finite set $D\subseteq \monster$ such that,    for every $a\in \phi(\monster)$, either there is $d\in D$ such that $a\le d$, or  $G_a\subseteq \phi(\monster)$.
\end{defin}
\begin{pr}\label{pr:nounary}
If $T$ is a purely binary cone-expansion of $\tdmt$, then every formula in one free variable is tame.
\end{pr}
\begin{proof}
  It is clear that every atomic and negated atomic $\phi(x)\in \lmt(\monster)$ is tame.
Fix a point $c$, and consider  $\phi(x)\coloneqq R_j(x,c)$; if $a\in \phi(\monster)$, by assumption we also have $\phi(b)$ for every $b>a$, hence $G_a\subseteq \phi(\monster)$. Consider now $\phi(x)\coloneqq \neg R_j(x,c)$, and let $D=\set c$. Suppose that $a\centernot \le c$. If  $a\parallel c$ and $\phi(a)$ holds, we can argue as above, so assume that $a>c$. For any $b\ge a$ we have in particular  $b> c$, hence $\phi(b)$ holds by assumption and  $G_a\subseteq \phi(\monster)$; therefore  $\neg R_j(x,c)$ is tame. The formula $R_j(x,x\meet c)$ and its negation are tame, because $R_j(x,x\meet c)$ is always false. As for the formula $\phi(x)\coloneqq R_j(x\meet c_0, x\meet c_1)$, take $D=\set{c_0,c_1}$. If $a\centernot \le c_0\land  a\centernot \le c_1$, then for every $b>a$ and $i<2$ we have $a\meet c_i=b\meet c_i$, hence $b>a\implica (\phi(a)\coimplica \phi(b))$, proving tameness of both $\phi(x)$ and $\neg \phi(x)$.   Since the same arguments apply to $R_j(c,x)$, $R_j(x\meet c,x)$, and their negations, we conclude that  all atomic and negated atomic formulas are tame.

  Tame formulas in the variable $x$ are easily seen to be closed under conjunctions and disjunctions: if $D_\phi$ and $D_\psi$  witness tameness of $\phi(x)$ and $\psi(x)$ respectively, then $D_\phi\cup D_\psi$ witnesses tameness of both $\phi(x)\land \psi(x)$ and $\phi(x)\lor \psi(x)$. By quantifier elimination, we have the conclusion.
\end{proof}

\begin{co}\label{co:restrinv}
  If $T$ is a purely binary cone-expansion of $\tdmt$ and $p\in \invtypes(\monster)$, then $(p\restr \lmt)\in \invtypes(\monster\restr \lmt)$. 
\end{co}
\begin{proof}
  Similarly to the final step in the previous proof, we see by taking unions of witnesses that, if $\Phi(x)$ is a small disjunction of small types, then it satisfies the analogue of tameness where we allow $D$ to have size $\abs{\Phi}$. By saturation, if $\Phi(\monster)$ is linearly ordered, it must be bounded.

Suppose now that  $p\in \invtypes(\monster, A)$, and let $q$ be a $1$-subtype of $p$. By invariance, $C_q$ is the set of realisations of a disjunction of $1$-types over $A$. By what we just proved, $C_q$ is bounded.  Keeping this in mind, along with the facts that the $R_j$ are on open cones and that $T$ eliminates quantifiers, we see by inspecting the possibilities for $d_q\phi$, where $\phi(x;w)$ is an $\lmt$-formula with $\abs x=1$,  that $q\restr \lmt$ must be invariant as well. By Remark~\ref{rem:invariantdim1} and Theorem~\ref{thm:expdmt}, this is enough.
\end{proof}

\begin{ass}
From now on, unless we say that $T$ is arbitrary, we work in  a purely binary cone-expansion  $T$ of $\tdmt$, in a language $L=\lmt\cup\set{R_j\mid j\in J}$.
\end{ass}

We saw in Theorem~\ref{thm:dmt} that, in $\tdmt$, domination-equivalence classes of invariant $1$-types correspond to either new open cones above existing points, or to certain cuts in linearly ordered subsets of $\monster$. In what follows, restrictions of invariant $1$-types to $\lmt$, which are still invariant by Corollary~\ref{co:restrinv},  will play a special role; we therefore introduce some terminology for these cones and cuts. 
\begin{defin}
Let $p,q\in \invtypes_1(\monster)$ be nonrealised, and suppose that $(p\restr \lmt)\domeq (q\restr \lmt)$ in $\tdmt$. If these restrictions are of kind (II), in a new open cone above $g\in \monster$, we say that $p,q$ have the same \emph{sprout}, and that each of them \emph{sprouts from $g$}. If the restrictions are of another kind, we say that $p,q$ have the same \emph{graft}.
\end{defin}
So, in Theorem~\ref{thm:dmt}, $X$ corresponds to the set of grafts, and there is a copy of $\mathbb N$ for each sprout. The reason behind the choice of terminology should be clear from Figure~\ref{figure:dmtkindsoftypes}.

\begin{lemma}\label{lemma:stillidempotent}
Let $p(x), q(y)\in \invtypes(\monster)$. Denote by $q\restr i$ the restriction of $q$ to the variable $x_i$, and similarly for $p$. If, for all $i<\abs y$ and $i'<\abs x$, the types $q\restr i$ and $p\restr i'$ have the same graft, then  $p\domeq q$.  
\end{lemma}

\begin{proof}
  As the roles of $p$ and $q$ are symmetric,  it is enough to prove $p\doms q$. By assumption and Theorem~\ref{thm:dmt} we have $(p\restr \lmt)\doms (q\restr \lmt)$,  witnessed by some $r'$ over a small set $A$, and all coordinates of  $p$ and $q$ have the same cut $C$. Recall that $C$ must be bounded by Corollary~\ref{co:restrinv} and the last point of  Lemma~\ref{lemma:cutremarks3}. Up to enlarging $A$, we may assume that (cf.~Lemma~\ref{lemma:meetclosedmodulofinite3})
  \begin{enumerate}
  \item there is  $a\in \monster$ such that $a>C$ and, if $C$ has a maximum $g$,   such that $a$ is in the same open cone above $g$ of each coordinate of $p$ and $q$; and
  \item $\monster y$ is closed under meets and $(\dcl^{\lmt}(\monster y)\setminus \monster)\subseteq \dcl^{\lmt}(A y)$.
  \end{enumerate}
  Furthermore, up to adjoining to $y$ finitely many points of $\dcl^{\lmt}(Ay)$, we may assume  $\dcl^{\lmt}(\monster y)\setminus \monster=y$. By the second point of Lemma~\ref{lemma:cutremarks3}, this assumption does not break the hypothesis that all points of $y$ have the same graft. Fix any $r\in S_{pq}(A)$ extending $r'$, and recall that $p\cup r\proves (q\restr \lmt)\cup (q\restr A)$.  By quantifier elimination and our assumptions on $y$, we are only left to deal with the formulas $R_j(y_{i},f)$ and $R_j(f, y_{i})$,  where $f\in \monster$ and $i< \abs y$.  We have three possibilities for $y_{i}\meet f$. If  $y_{i}\meet f=y_{i}$, then $y_{i}\le f$. If instead $y_{i}\meet f\in \monster$, then there must be a point of $\monster$ in the same open cone as $y_{i}$ above $y_{i}\meet f$, because otherwise $(q\restr i)\restr \lmt$ would be of kind (II). In the only other possible case, which can only arise if $(q\restr i)\restr \lmt$ is of kind (IIIa) or (IIIb), it is easy to see that $f$ must be in the same open cone above $y_{i}\meet f$ as $a$. In each case, since the $R_j$ are on open cones, $a\in A$, and $r\in S_{pq}(A)$, the partial type $p\cup r$ decides whether $R_j(y_{i},f)$ and $R_j(f, y_{i})$ hold, and we are done.
\end{proof}
\begin{rem}
  The set of grafts of types in $\invtypes_1(\monster)$ can be identified with that of grafts of types in $\invtypes_1(\monster\restr \lmt)$. 
\end{rem}
\begin{proof}
  The natural map from the former to the latter,   well-defined by Corollary~\ref{co:restrinv}, is injective by definition of graft, and is surjective because, since $T$ is a purely binary cone-expansion of $\tdmt$, if $p\in \invtypes_1(\monster\restr \lmt)$ is of kind (Ia) or (Ib), then $p$ implies a unique type in $S_1(\monster)$, easily seen to be invariant.
\end{proof}
\begin{lemma}\label{lemma:stillorthgonal}
Let $\bla p0,n\in \invtypes(\monster)$ be such that $\Phi\coloneqq\bigcup_{i\le n} (p_i(x^i)\restr \lmt)$ is a complete type in $\tdmt.$ Then $\bigcup_{i\le n} p_i(x^i)$ is a complete type in $T$.
\end{lemma}
\begin{proof}
  Let $b^i\models p_i$. In order for $\Phi$ to be complete in $\tdmt$, given $i< i'\le n$, no coordinate of $p_i$  can have the same graft as a coordinate of $p_{i'}$: if this was the case for $x^i_0$ and $x^{i'}_0$, say, then there would be $a\in \monster$ such that $\Phi$ does not decide whether $x^i_0\meet a=x^{i'}_0\meet a$ holds. Similarly, no coordinate of $p_i$, say $x^i_0$ again, can have the same sprout as a coordinate of $p_{i'}$, say $x^{i'}_0$, otherwise $\Phi$ does not decide whether $x^i_0=x^{i'}_0$ holds. It follows from this observation and  Lemma~\ref{lemma:cutremarks3} that $\dcl^{\lmt}(\monster b^0,\ldots, b^n)=\bigcup_{i\le n} \dcl^{\lmt}(\monster b^i)$. Therefore, we only need to show that, for each  $i<i'\le n$, each $y\in \dcl^{\lmt}(\monster x^i)\setminus \monster$ , and each $z\in \dcl^{\lmt}(\monster x^{i'})\setminus \monster$, every formula of the form $R_j(y,z)$ is decided by $p_i(x^i)\cup p_{i'}(x^{i'})$.  
  Since the $R_j$ are on open cones, it is enough to show that at least one between $y$ and $z$ must be in the same open cone above $y\meet z$ as a point of $\monster$.  Again because $\Phi$ is complete, $y$ and $z$ cannot have the same graft, nor the same sprout.

  We have two cases. Suppose first that $\Phi\proves y\meet z=g$ for some $g\in \monster$. This happens for example if $\tp(y/\monster)$  sprouts from $g$ and the graft of $\tp(z/\monster)$ is in an existing open cone above $g$, or if none of $C_y^\monster$ and $C_z^\monster$ is included in the other. Then, at least one between $y$ and $z$ must be in an open cone above $g$ represented in $\monster$, because otherwise both would be sprouting from $g$, contradicting completeness of $\Phi$.

  If instead  $\Phi\proves \text{``}y\meet z\notin \monster\text{''}$  then, up to swapping  $y$ and $z$, we must have $\Phi\proves\text{``}C_z^\monster\subsetneq C_y^\monster\text{''}$, because otherwise $y$ and $z$ have the same graft.  It follows that for some $a\in \monster$ we have  $\Phi\proves y>a>y\meet z$, and in particular $y$ is in the same open cone above $y\meet z$ as $a$.
\end{proof}

Recall that  a sort $Y$ of a multi-sorted   $\monster$ is said to be \emph{stably embedded} iff, whenever $D\subseteq \monster^{m}$ is definable,  then $D\cap Y^m$ is definable with parameters from $Y$, in the sense that it is definable with parameters when we view  $Y$ as a structure on its own, the atomic relations being the traces on  $Y$ of $\emptyset$-definable relations of $\monster$.  It is easy to obtain a proof of the following fact; the reader may find one in~\cite[Proposition~2.3.31]{mythesis}. 
\begin{fact}[$T$ arbitrary]\label{fact:stabemb}
  Let $Y$ be a stably embedded sort of $\monster$. There is an embedding of posets $\invtildeof{Y} \into \invtilde$. This embedding is a $\wort$-homomorphism, a $\nwort$-homomorphism,  and,  if $\otimes$ respects $\doms$,  an embedding of monoids.
\end{fact}
For  $g\in \monster$, denote by $O_g$ the  set of open cones above $g$ equipped with the $\set{R_j\mid j\in J}$-structure induced by $\monster$. This may be regarded as an imaginary  sort of the expansion of $\monster$ obtained by naming the point $g$, and   each type $p\in S_n(O_g)$  may be seen as the pushforward under the projection map of a suitable $q\in S_n(\monster)$ with all non-realised coordinates sprouting from $g$. Since $T$ eliminates quantifiers, $q$ is axiomatised by its quantifier-free part. It follows easily 
that the same is true of $p$, and therefore $\Th(O_g)$ eliminates quantifiers in a binary language, hence is (weakly) binary.  By Fact~\ref{fact:wbinwd}, $I_g\coloneqq\invtildeof{O_g}$ is well-defined.

\begin{thm}\label{thm:invtildexp}
  Let $T$ be a purely binary cone-expansion of $\tdmt$, and let $X$ be the set of grafts of types in $\invtypes_1(\monster)$. Then
  \[
    \invtilde\cong \pfin(X)\oplus\bigoplus_{g\in \monster} I_g
  \]
\end{thm}
\begin{proof}
Recall that  by  Theorem~\ref{thm:expdmt} $\invtilde$ is well-defined, and that  by Corollary~\ref{co:restrinv} taking restrictions to $\lmt$ preserves invariance.    By Lemma~\ref{lemma:stillorthgonal}, $\invtilde$ is generated by the $\domeq$-classes of those types $p$ where  coordinates of $p$ have all the same graft, or have all the same sprout.  If all coordinates of $p$ have the same graft, by Lemma~\ref{lemma:stillidempotent} $p$ is domination-equivalent to any $1$-type with such a graft, and by using Lemma~\ref{lemma:stillorthgonal} a second time we see that $\pfin(X)$ embeds in $\invtilde$. Again by Lemma~\ref{lemma:stillorthgonal}, $\invtilde=\pfin(X)\oplus\bigoplus_{g\in \monster}\tilde I_g$, where $\tilde I_g$ is the monoid of types whose every  coordinate sprouts from $g$.

  We are only left to show that $\tilde I_g\cong I_g$. Fix $g\in \monster$. Since $\invtilde$ does not change after naming a small number of constants, we may add to $L$  a constant symbol to be interpreted as  $g$. For the time being, we also adjoin to the language a sort for $O_g$ and its natural projection map $\pi_g$. Call the resulting structure $\monster_g$.  Clearly $\monster$ is stably embedded in $\monster_g$, so by Fact~\ref{fact:stabemb} we have an embedding $\invtilde\into\invtildeof{\monster_g}$. Similarly, $O_g$ is stably embedded, hence  $\invtildeof{O_g}=I_g$ embeds in $\invtildeof{\monster_g}$.   Let $p$ be a type with all coordinates sprouting from $g$ (different coordinates might be in different open cones), and let $q$ be the pushforward of $p$ along $\pi_g$. Clearly $p\doms q$, and if we show that $q\doms p$ we may simply conclude by discarding the sort $O_g$ and forgetting about the new constant symbol. That $q\doms p$ is easily seen to be witnessed by any $r$ containing all the formulas  $y_i=\pi_gx_i$ for $i<\abs x$: the only information lost when taking the projection concerns points in the same new open cone, but this information is in $r$.  For instance, if  $x_0\meet x_1>g$, we need to recover whether $R_j(x_0,x_1)$ holds, and whether any inequality holds between $x_0$ and $x_1$. More generally, the information we need to recover is implied by $p\restr \emptyset$, which is included in $r$ by Definition~\ref{defin:domination}.
\end{proof}

\end{document}